\documentclass{amsart}

\usepackage{amsmath,amsthm, amssymb}
\usepackage[all]{xy}
\usepackage{enumerate}
\usepackage{wasysym}
\usepackage{graphicx,color}
\usepackage{datetime}

\newcommand{\R}{\mathbb R}
\newcommand{\Z}{\mathbb{Z}}

\newcommand{\ixi}{i_\xi}

\DeclareMathOperator{\lef}{Lef}
\newcommand{\rel}{\mathcal{R}}

\newtheorem{theorem}{Theorem}[section]

\newtheorem{corollary}[theorem]{Corollary}
\newtheorem{proposition}[theorem]{Proposition}
\theoremstyle{definition}

\theoremstyle{remark}
\newtheorem{remark}[theorem]{Remark}

\numberwithin{equation}{section}

\title[A non-Sasakian Lefschetz $K$-contact manifold]{A non-Sasakian Lefschetz $K$-contact manifold of Tievsky type}

\author[B. Cappelletti-Montano]{Beniamino Cappelletti-Montano}
 \address{Dipartimento di Matematica e Informatica, Universit\`a degli Studi di
 Cagliari, Via Ospedale 72, 09124 Cagliari, Italy}
 \email{b.cappellettimontano@gmail.com}

\author[A. De Nicola]{Antonio De Nicola}
 \address{CMUC, Department of Mathematics, University of Coimbra, 3001-501 Coimbra, Portugal}
 \email{antondenicola@gmail.com}

 \author[J.~C. Marrero]{Juan Carlos Marrero}
 \address{Unidad Asociada ULL-CSIC ``Geometr{\'\i}a Diferencial y Mec\'anica Geo\-m\'e\-tri\-ca''
Departamento de Matem\'aticas, Estad{\'\i}stica e Investigaci\'on Operativa, Facultad de Ciencias, Universidad de La Laguna, La Laguna, Tenerife, Spain}
 \email{jcmarrer@ull.edu.es}

\author[I. Yudin]{Ivan Yudin}
 \address{CMUC, Department of Mathematics, University of Coimbra, 3001-501 Coimbra, Portugal}
 \email{yudin@mat.uc.pt}

\subjclass[2000]{Primary 53C25, 53D35 }

\thanks{This work was partially supported by CMUC -- UID/MAT/00324/2013, funded by the Portuguese
 Government through FCT/MEC and co-funded by the European Regional Development Fund through the Partnership Agreement PT2020 (A.D.N. and I.Y.),
 by MICINN (Spain) grants MTM2012-34478 (A.D.N. and J.C.M.) and MTM2015-64166-C2-2-P (J.C.M.),
 by Prin 2010/11 -- Variet\`{a} reali e complesse: geometria, topologia e analisi armonica -- Italy (B.C.M.), and by the exploratory research project in the frame of Programa Investigador FCT IF/00016/2013 (I.Y.). The authors would like to thank the referee for his/her valuable comments which helped to improve the manuscript.}

\begin{document}

\begin{abstract}
We find a family of five dimensional completely solvable compact manifolds that constitute the first examples of  $K$-contact manifolds which satisfy the Hard Lefschetz Theorem and have a model of Tievsky type just as Sasakian manifolds but do not admit any Sasakian structure.
\end{abstract}

\maketitle

\section{Introduction}

It is well known that Sasakian geometry is closely related to K\"{a}hler geometry. Indeed, on the one hand, an odd dimensional Riemannian manifold $(M,g)$ is Sasakian if and only if its metric cone $(M\times \mathbb{R}^{+},r^2 g + dr^2)$ is a K\"{a}hler manifold. On the other hand, the Reeb vector field  of a Sasakian manifold generates a transversely K\"{a}hler foliation and thus, in particular, if the foliation is regular the leaf space is a K\"{a}hler manifold.

A natural problem in K\"{a}hler geometry is the study of obstructions to the existence of a K\"{a}hler structure on a given compact symplectic manifold. Similarly, in Sasakian geometry a problem of interest is the study of obstructions to the existence of a Sasakian structure on a given compact contact manifold. On any Sasakian manifold the Reeb vector field is Killing. Hence the only interesting case to consider is when a contact manifold has Killing Reeb vector field with respect to some metric. The problem of constructing explicit examples of $K$-contact manifolds with no Sasakian structures arises.

The study of various obstructions to the existence of a Sasakian metric on a given compact contact manifold intensified after the appearance of the influential book \cite{BoGa} by Ch.~Boyer and K.~Galicki on Sasakian geometry. In Chapter 7 of this book the authors presented a remarkable list of problems on the topology of Sasakian and, more generally, $K$-contact manifolds. Several obstructions are now known.

For  a contact manifold $(M,\eta)$, we write $H_{B}^{\ast}(M)$ for the basic cohomology with respect to the foliation induced by the Reeb vector field. In 2008 Tievsky proved in his Ph.D. thesis \cite{Ti} that in order to admit a compatible Sasakian structure the de Rham algebra $(\Omega^*(M), d)$ of a compact contact manifold $(M,\eta)$ has to be quasi-isomorphic as CDGA to an elementary Hirsch extension of  $H_B(M)$.
Namely, one has the model $(\mathcal{T}^{\ast}(M),d)$ given by
\begin{equation}\label{tievskymodel}
\mathcal{T}^{\ast}(M):=H_{B}^{\ast}(M)\otimes\mathbb{R}[y]/(y^2), \quad \quad d([\alpha]_{B} + [\beta]_{B}y):=[\beta\wedge d\eta]_B,
\end{equation}
where we put $\deg(y):=1$,  $\alpha$ is a basic $p$-form and $\beta$ a basic $(p-1)$-form. We will refer to a compact contact manifold $(M,\eta)$ admitting the model $(\mathcal{T}^{\ast}(M),d)$ as to a manifold of \emph{Tievsky type}.

Using Tievsky's result, we proved in \cite{CaDeMaYu} that a compact contact nilmanifold of dimension $2n+1$ is of Tievsky type if and only if it is a compact quotient of the  Heisenberg group $H(1,n)$ of dimension $2n+1$ by a co-compact discrete subgroup. In particular, any compact nilmanifold carries a Sasakian structure if and only if it is of this type. To the knowledge of the authors, so far, no examples of non-Sasakian compact contact manifolds of Tievsky type have been found.

Another  obstruction to the existence of Sasakian structures, recently discovered in \cite{CaNiYu}, is provided by the so-called \emph{Hard Lefschetz Theorem} for Sasakian manifolds. Let $(M^{2n+1},\eta)$ be a compact contact manifold with Reeb vector field $\xi$.
Consider the following relation between $H^p(M)$ and $H^{2n+1-p}(M)$:
 \begin{equation*}
	\!\rel_{\lef_p} \!\!= \left\{ \left(\, [ \beta ]\, , \, [ \eta\wedge (d \eta)^{n-p}\!\wedge\beta
	] \, \right) \, \middle|\,  \beta \in \Omega^p(M),\; d\beta = 0,\;  \ixi\beta =0,\; 	(d \eta)^{n-p+1}\!\wedge\beta =0
 \right\},\phantom{M}	
\end{equation*}
for $p \leq n$. In \cite{CaNiYu} it is proved that if $(M^{2n+1},\eta)$ admits a compatible Sasakian metric,  then $\rel_{\lef_p}$ is the graph of an isomorphism $\lef_p: H^p(M)\longrightarrow H^{2n+1-p}(M)$. In such a case we say that  $(M,\eta)$ is a \emph{Lefschetz contact manifold}.
In any Lefschetz contact manifold the $p$-th Betti number is even for $p$ odd with $1\leq p\leq n$.  Note that for Sasakian manifolds this fact  was known since \cite{Fu}. We point out that the  property given by the Hard Lefschetz Theorem is stronger than the above mentioned obstruction in terms of Betti numbers. Namely, in \cite{CaDeMaYu2} we constructed examples of compact $K$-contact manifolds satisfying these restrictions on the Betti numbers but for which the Hard Lefschetz Theorem does not hold. These examples in dimensions 5 and 7 where not simply-connected.
In \cite{Ch}, X.~Chen found restrictions on the possible fundamental groups of compact Sasakian manifolds and used these to produce (non-simply-connected) compact manifolds which are $K$-contact and not Sasakian.

The first examples of compact simply-connected $K$-contact manifolds of any dimension $\geq 9$ which do not admit  Sasakian structures where presented by B.~Hajduk and A.~Tralle in \cite{HaTr}. In fact, in \cite{HaTr}, the authors prove that the third Betti number of these manifolds is odd. An alternative approach to construct examples of compact simply-connected non-Sasakian $K$-contact manifolds in dimension $\geq 9$ can be found in \cite{Li}. More subtle techniques of homotopy theory combined with symplectic surgery are employed by V.~Mu{\~n}oz and A.~Tralle in \cite{MuTr} to present examples of compact simply-connected $K$-contact non-Sasakian  manifolds in dimension $7$.

On the other hand, in \cite{BiFeMuTr} I.~Biswas et al.  prove that the Massey products of order higher than three are zero on Sasakian manifolds and thus they provide a computable obstruction to the existence of Sasakian structures. Using this result, they show a new method of constructing families of examples of compact simply-connected $K$-contact manifolds with no Sasakian structures. They also present the first examples of compact simply-connected Sasakian manifolds with non-vanishing triple Massey products, which are henceforth non-formal.
We also quote the very recent paper \cite{MuRoTr} which deals with the 5-dimensional case and provides the first example of a closed 5-manifold $M$ with $H_1(M,\Z)=0$ which is $K$-contact but carries no semi-regular Sasakian structure.

The aim of the paper is to give the first example of a compact $5$-dimensional $K$-contact manifold that satisfies all known necessary conditions for compact manifold to admit Sasakian structures, even the same type of real homotopy model, and yet it does not admit any Sasakian structure.
More precisely, we present a family of 5-dimensional compact completely solvable $K$-contact formal manifolds of Tievsky type with no Sasakian structure.
We give two descriptions of our family of examples. Namely, in Section \ref{compact-solv} for each real number $p$ we construct a 5-dimensional compact $K$-contact completely solvable manifold $\widehat{G(p)}$ starting from a nilpotent Lie group of dimension $3$ and then we compute its de Rham cohomology. In Section \ref{second-de} we show that each $\widehat{G(p)}$ can also be constructed from a completely solvable Lie group of dimension $3$ and has the structure of a principal $S^1$-bundle over a symplectic compact solvmanifold of dimension $4$ which  was considered in \cite{FeGr}. This second description allows us to prove that the manifold $\widehat{G(p)}$ possesses the claimed properties.

\section{ A first description of the compact solvmanifold }\label{compact-solv}

\subsection{The completely solvable Lie group $G(p)$}\label{completely-solvable}
Let $H(1, 1)$ be the Heisenberg group of dimension $3$ which may be described as the group of real matrices of the form
\[
\left(
\begin{array}{rcl}
1 & x &    z \\
0 & 1  &  y\\
0  & 0  &   1
\end{array}
\right).
\]
The group $H(1, 1)$ is a connected simply-connected nilpotent Lie group. We will denote by $(x, y, z)$ the standard global coordinates on $H(1,1)$.

Then, our connected simply-connected Lie group of dimension $5$ is the semi-direct product
\[
G(p) = (H(1, 1) \rtimes_{\psi} \mathbb{R})\rtimes_{\phi}\mathbb{R}  ,
\]
where  $\psi: \mathbb{R} \to Aut(H(1,1))$ and $\phi: \mathbb{R} \to Aut(H(1, 1)\rtimes_{\psi}\mathbb{R}u )$ are the representations defined by
\begin{equation}
\psi(u)(x, y, z)  =  (e^{pu}x, e^{-pu}y, z), \; \; \;
\phi(t)(x, y, z,u)  =  (x, y, z + tu, u).
\label{psi-phi}
\end{equation}
Here $ p \in \mathbb{R} \setminus \{0\}$ is a parameter and $(x, y, z, u, t)$ are the standard coordinates on $G(p) =\R^5$.

The Lie group $H(1,1) \rtimes_{\psi} \R$ was considered previously in \cite{Has} (see page 756  in \cite{Has}).

From (\ref{psi-phi}) we deduce that the multiplication in $G(p)$ is given by
\begin{equation}
\begin{array}{rcl}
(x, y, z, u, t) (x', y', z', u', t') & = & (x + e^{pu}x',  y +  e^{-pu}y',\\
 && z + z' + x e^{-pu}y' + tu', u+u', t + t').
\end{array}
\label{multiplication-G(p)}
\end{equation}

The $1$-forms
\begin{equation}\label{left-inv-1-forms-G(p)}
 \alpha = e^{-pu} dx, \; \; \beta = e^{pu}dy,  \; \;  \theta = du, \; \; \gamma = dt, \; \; \eta = dz - x dy - t du
\end{equation}
give a basis of left-invariant $1$-forms on $G(p)$.
Note that
\begin{equation}\label{structure-equations}
  d\alpha = p \alpha \wedge \theta, \; \; d\beta = -p \beta \wedge \theta, \; \; d \theta = 0, \; \; d\gamma = 0, \; \; d\eta = -\alpha \wedge \beta - \gamma \wedge \theta.
\end{equation}
Next, we consider the dual basis $\{A, B, U, R, \xi \}$ of left-invariant vector fields. It follows that
\begin{equation}\label{basis-G(p)}
 A = e^{pu}\frac{\partial}{\partial x}, \; \; B = e^{-pu}\left(\frac{\partial}{\partial y} + x \frac{\partial}{\partial z}\right), \; \; U = \displaystyle \frac{\partial}{\partial u} + t \frac{\partial}{\partial z}, \; \; R = \displaystyle \frac{\partial}{\partial t}, \; \; \xi = \frac{\partial}{\partial z},
\end{equation}
and, therefore,
\begin{equation}\label{structure-eqs-G(p)}
[A, B] = \xi, \; \; [A, U] = -pA, \; \; [B, U] = pB, \; \; [R, U] = \xi,
\end{equation}
the rest of the basic Lie brackets being zero.

\begin{remark}\label{rmk:non-nilpotent}
This implies that $G(p)$ is a non-nilpotent, but completely solvable Lie group.
\end{remark}
\subsection{A co-compact discrete subgroup of $G(p)$}\label{co-com-discrete-sub}
In this section, we will describe a co-compact discrete subgroup of $G(p)$.

In fact, we will prove the following result
\begin{proposition}\label{discrete-subgroup-G(p)}
Let $p$ be a real number, $p \neq 0$, and $N$ a unimodular $2\times2$ matrix with integer entries and eigenvalues $e^p$ and $e^{-p}$. Suppose that $(x_0, x_1)$ (resp. $(y_0, y_1)$) is an eigenvector for the eigenvalue $e^p$ (resp. $e^{-p}$). Then, one may find a co-compact discrete subgroup $\Gamma(p)$ of $G(p)$ of the form
\[
\Gamma(p) =  ( \Gamma_N\rtimes_{\psi} \mathbb{Z})\rtimes_{\phi} z_2 \mathbb{Z},
\]
where $z_2 = x_1 y_0 - x_0 y_1$ and $\Gamma_N$ is a co-compact discrete subgroup of $H(1,1)$ which satisfies the following conditions:
\begin{enumerate}
\item
$\Gamma_N$ is invariant under the restriction to $\mathbb{Z}$ of the representation $\psi$ and
\item
$\Gamma_S = \Gamma_N\rtimes_{\psi} \mathbb{Z} $ is a co-compact discrete subgroup of the Lie group $S = H(1,1)\rtimes_{\psi}\mathbb{R} $ which is invariant under the restriction to $z_2\mathbb{Z}$ of the representation $\phi$.
\end{enumerate}
\end{proposition}
\begin{proof}
In order to obtain co-compact discrete subgroups of the Lie group $S$, a general construction was developed in \cite{Has} (see pages 756 and 757 in \cite{Has}).

We will see that it is possible to choose a co-compact discrete subgroup $\Gamma_S$ of $S$ of the form $\Gamma_N \rtimes_{\psi} \mathbb{Z}$, with $\Gamma_N$ a co-compact discrete subgroup of $H(1, 1)$ such that $\Gamma_S$ is invariant under the restriction to $z_2\mathbb{Z}$ of the representation $\phi$. In order to obtain $\Gamma_S$, we will follow the construction in \cite{Has}.

Suppose that
\[
N =
\left(
\begin{array}{rcl}
n_{00} & n_{01}  \\
n_{10} & n_{11}
\end{array}
\right)
\in Sl(2, \mathbb{Z}).
\]
Then, we take the matrices in $H(1, 1)$
\begin{equation}
h_0 =
\left(
\begin{array}{rcl}
1 & x_0 &    z_0 \\
0 & 1  &  y_0\\
0  & 0  &   1
\end{array}
\right),
\; \, \; \; h_1 =
\left(
\begin{array}{rcl}
1 & x_1 &    z_1 \\
0 & 1  &  y_1\\
0  & 0  &   1
\end{array}
\right)
\label{h-0-h-1}
\end{equation}
with $z_0, z_1 \in \mathbb{R}$ which will be fixed  later.

We have that
\[
h_2(h_0 h_1) = h_1 h_0,
\]
with
\begin{equation}
h_2 =
\left(
\begin{array}{rcl}
1 & 0 &    z_2 \\
0 & 1  &  0 \\
0  & 0  &   1
\end{array}
\right) \; \; \mbox{ and } z_2 = x_1 y_0 - x_0 y_1 \neq 0,
\label{h-2}
\end{equation}
(note that $\{(x_0, y_0), (x_1, y_1)\}$ is a basis of $\mathbb{R}^2$).

Thus, we deduce that
\begin{equation}\label{Commutation}
h_2 h_0 = h_0 h_2, \; \; h_1 h_2 = h_2 h_1, \; \; h_1^{m_1} h_0^{m_0} = h_2^{m_0 m_1}h_0^{m_0} h_1^{m_1}, \; \; \; \mbox{ for } m_0, m_1 \in \mathbb{Z}.
\end{equation}
Denote by $\Gamma_{N}$ the subgroup of $H(1, 1)$ which is generated by $h_0$ and $h_1$.

From (\ref{Commutation}), it follows that
\begin{equation}\label{Gamma-N}
\Gamma_{N} = \{ h_0^{m_0} h_{1}^{m_1} h_2^{m_2} \; / \; m_0, m_1, m_2 \in \mathbb{Z} \}\subseteq H(1,1).
\end{equation}
In fact, the map
\[
\mathbb{Z}^3 \to \Gamma_{N}, \; \; (m_0, m_1, m_2) \to h_0^{m_0}  h_{1}^{m_1} h_2^{m_2}
\]
is a group isomorphism, where we consider the multiplication
\[
(m_0, m_1, m_2)(m'_0, m'_1, m'_2) = (m_0 + m'_0, m_1 + m'_1, m_2 + m'_2 + m_1 m'_0)
\]
on $\mathbb{Z}^3$.

Next, we need to ensure that the subgroup $\Gamma_N$ is invariant under the restriction to $\mathbb{Z}$ of the representation $\psi$. This will produce some restrictions on the real numbers $z_0$ and $z_1$.

In fact, using that the eigenvalues of $N$ are different from $1$, one may choose $z_0, z_1 \in \mathbb{R}$ such that
\[
\psi(1)(h_0) = h_0^{n_{00}} h_{1}^{n_{01}} h_{2}^r, \; \; \psi(1)(h_1) = h_0^{n_{10}} h_{1}^{n_{11}} h_{2}^s
\]
with $r, s \in \mathbb{Z}$.

This implies that $\Gamma_{N}$ is invariant under the restriction to $\mathbb{Z}$ of $\psi$. Thus, we have that $\Gamma_S = \Gamma_N \rtimes_{\psi} \mathbb{Z}$ is a co-compact discrete subgroup of the Lie group $S =  H(1, 1)\rtimes_{\psi}\mathbb{R} $.

On the other hand, from (\ref{psi-phi}) we deduce that the multiplication in the Lie group $S$ is given by
\begin{equation}\label{multiplicacion-H}
(x, y, z, u) (x', y', z', u') =  (x + e^{pu}x', y + e^{-pu}y', z + z' + x e^{-pu}y', u+u').
\end{equation}
Therefore, using again (\ref{psi-phi}), it follows that
\[
\phi(t)(x, y, z, u) = (x, y, z, u) (0, 0, tu, 0).
\]
Consequently, if $m\in \mathbb{Z}$ and $(x, y, z, u) \in \Gamma_S$ then we deduce that
\[
\phi(z_2 m)(x, y, z,u) = (x, y, z,u) (h_2^{mu}, 0) \in \Gamma_{S}.
\]
In other words, $\Gamma_S$ is invariant under the restriction to $z_2 \mathbb{Z}$ of the representation $\phi$. This implies that
\begin{equation}\label{Gamma-p}
\Gamma(p) = \Gamma_{S} \rtimes_{\phi} z_2 \mathbb{Z} = ( \Gamma_{N}\rtimes_{\psi}\mathbb{Z})\rtimes_{\phi} z_2 \mathbb{Z}
\end{equation}
is a co-compact discrete subgroup of $G(p)$.
\end{proof}
\begin{remark}
The Lie group $S$ is completely solvable and the quotient manifold $S / \Gamma_S$ is a compact solvmanifold. In fact, $S / \Gamma_S$ is diffeomorphic to an Inoue surface of type $S^{+}$ (see page 757 in \cite{Has}).
\end{remark}

\subsection{De Rham cohomology of the compact solvmanifold $\widehat{G(p)} = G(p) / \Gamma(p)$}
\label{derham}
The Lie group $G(p)$ is completely solvable. So, in order to compute the de Rham cohomology of $\widehat{G(p)}$, we may use Hattori's theorem \cite{Ha} and we have that
\[
H^*(\widehat{G(p)}) \simeq H^*({\mathfrak g}^*),
\]
where ${\mathfrak g}$ is the Lie algebra of $G(p)$.

Thus, if we denote by  $\hat{\alpha}$, $\hat{\beta}$, $\hat{\theta}$, $\hat{\gamma}$ and $\hat{\eta}$ the $1$-forms on $\widehat{G(p)}$ which are induced by the left-invariant $1$-forms $\alpha$, $\beta$, $\theta$, $\gamma$  and $\eta$, respectively, we deduce the following result.
\begin{corollary}
	\label{derham-cor}
The de Rham cohomology groups of $\widehat{G(p)}$ are the vector spaces
\begin{equation}
\begin{array}{rcl}
H^0(\widehat{G(p)}) & = & \langle 1 \rangle, \\
H^1(\widehat{G(p)}) & = & \langle [\hat{\theta}], [\hat{\gamma}] \rangle, \\
H^2(\widehat{G(p)}) & = & \langle [\hat{\theta} \wedge \hat{\gamma} ] \rangle = \langle [\hat{\theta} \wedge \hat{\gamma} - \hat{\alpha} \wedge \hat{\beta}] \rangle, \\
H^3(\widehat{G(p)}) & = & \langle [\hat{\eta} \wedge (\hat{\theta} \wedge \hat{\gamma} - \hat{\alpha} \wedge \hat{\beta}) ] \rangle, \\
H^4(\widehat{G(p)}) & = & \langle [\hat{\eta} \wedge \hat{\gamma} \wedge \hat{\alpha} \wedge \hat{\beta}], [\hat{\alpha} \wedge \hat{\beta}\wedge \hat{\theta} \wedge \hat{\eta}] \rangle, \\
H^5(\widehat{G(p)}) & = & \langle [ \hat{\alpha} \wedge \hat{\beta} \wedge \hat{\theta} \wedge \hat{\gamma} \wedge \hat{\eta}] \rangle.
\end{array}
\label{deRham-compact-solv}
\end{equation}
\end{corollary}

\section{A second description of the contact compact solvmanifold $\widehat{G(p)}$}\label{second-de}
In this section, we will present a different description of the compact solvmanifold $\widehat{G(p)}$ as a principal $S^1$-bundle over a symplectic compact solvmanifold $(K / \Gamma_K) \times S^1 $ of dimension $4$. In fact, from (\ref{structure-equations}), we deduce that $\hat{\eta}$ defines a contact structure on $\widehat{G(p)}$ and, in addition, we will see that $(K/\Gamma_{K})\times S^1 $ is just the space of orbits of the Reeb vector field associated with $\hat{\eta}$ and the symplectic structure on this space is induced by the contact structure of $\widehat{G(p)}$.

Define $K$ to be the connected simply-connected Lie group of dimension $3$ given by
\[
K = \R^2 \rtimes_{\zeta} \R
\]
where $\zeta: \mathbb{R} \to Aut(\mathbb{R}^2)$ is the representation given by
\[
\zeta(u)(x, y) = (e^{pu}x, e^{-pu}y),
\]
for $u \in \mathbb{R}$ and $(x, y) \in \mathbb{R}^2$ (see \cite{AuGrHa}).

Thus, if we consider the standard coordinates $(u, x, y)$ on $K$ then the multiplication is given by
\begin{equation}\label{multiplication-K}
(x, y, u) (x', y', u') =  (x + e^{pu}x', y + e^{-pu}y',u+u').
\end{equation}
We have that a basis of left-invariant $1$-forms is $\{ \alpha_{K}, \beta_{K}, \theta_{K} \}$, with
\begin{equation}\label{left-inv-1-forms-K}
\alpha_{K} = e^{-pu} dx, \; \; \beta_{K} = e^{pu}dy, \;\; \theta_{K} = du.
\end{equation}
Note that
\begin{equation}\label{structure-eqs-K}
d\alpha_K = p \alpha_K \wedge \theta_K, \; \; d\beta_K = -p \beta_K \wedge \theta_K, \;\; d \theta_{K} = 0.
\end{equation}
Now, we consider the dual basis $\{A_K, B_K, U_K \}$ of left-invariant vector fields on $K$. It follows that
\[
A_K = e^{pu}\frac{\partial}{\partial x}, \; \; B_K = e^{-pu} \frac{\partial}{\partial y}, \;\; U_K = \displaystyle \frac{\partial}{\partial u}
\]
and, thus,
\[
[A_K, U_K] = -pA_{K}, \; \; [B_K, U_K] = pB_K,
\]
the rest of the basic Lie brackets being zero.

Therefore, $K$ is a non-nilpotent completely solvable Lie group.

The group $K$ admits co-compact discrete subgroups (see \cite{BeGo}).
In fact, suppose that $N \in Sl(2, \mathbb{Z})$, with eigenvalues $e^p$ and $e^{-p}$, and that $(x_0, x_1)$, $(y_0, y_1)$ are eigenvectors of $N$ as in Section \ref{co-com-discrete-sub}. Then, we have that the integer lattice $\Gamma_{\mathbb{R}^2}$ on $\mathbb{R}^2$, which is generated by the basis $\{(x_0, y_0), (x_1, y_1)\}$, is invariant under the restriction to $\Z$ of $\zeta$ and
\begin{equation}\label{Gamma-K}
\Gamma_K = \Gamma_{\R^2} \rtimes_{\zeta} \Z
\end{equation}
is a co-compact discrete subgroup of $K$ (see \cite{BeGo}).

This implies that $\hat{K} = K / \Gamma_K$ is a compact solvmanifold.

Next, we consider the completely solvable Lie group $K \times \R$ and the discrete co-compact subgroup
\begin{equation}\label{discrete-base}
\Gamma_{K \times \R} =  \Gamma_K  \times z_2 \Z,
\end{equation}
with $z_2 = x_1 y_0 - x_0 y_1 \in \mathbb{R}$.
The corresponding compact solvmanifold is
\begin{equation}\label{description-base-space}
\displaystyle \frac{K \times \R}{\Gamma_{K\times \R}} = \hat{K}\times (\R / z_2 \Z) \simeq \hat{K} \times S^1.
\end{equation}

On the other hand, using (\ref{structure-eqs-K}), we have that
\begin{equation}\label{Omega}
 \Omega = -(\alpha_K \wedge \beta_K + dt \wedge \theta_K)
 \end{equation}
is a left-invariant symplectic $2$-form on $K\times \R$. Thus, $\Omega$ induces a symplectic form $\hat{\Omega}$ on $\hat{K}\times S^1$.
In fact, if $\hat{\omega}$ is the length element of $S^1$ and $\{ \widehat{\alpha_K}, \widehat{\beta_K}, \widehat{\theta_K}\}$ is the global basis of $1$-forms on $\hat{K}$ induced by the left-invariant $1$-forms $\alpha_K$, $\beta_K$ and $\theta_K$, respectively, then
\begin{equation}\label{hat-Omega}
\hat{\Omega} = -(\widehat{\alpha_K} \wedge \widehat{\beta_K} + \hat{\omega} \wedge \widehat{\theta_K}).
\end{equation}
Moreover, we may prove the following result
\begin{theorem}\label{regular-contact}
The contact compact solvmanifold $(\widehat{G(p)}, \hat{\eta})$ is regular, the orbit space of the Reeb vector field $\hat{\xi}$ is the compact solvmanifold $\hat{K}\times S^1$ and the symplectic structure on $\hat{K}\times S^1$ is just the $2$-form $\hat{\Omega}$ given by (\ref{hat-Omega}).
\end{theorem}
\begin{proof}
Using (\ref{multiplication-G(p)}) and (\ref{multiplication-K}), we deduce that the canonical projection
\begin{equation}\label{pi}
\pi: G(p) \to K \times \R, \; \; (x, y, z, u, t) \to (x, y, u, t)
\end{equation}
is a Lie group epimorphism. In fact, from (\ref{basis-G(p)}), it follows that $K\times \R$ is just the space of orbits of the left-invariant vector field $\xi$ on $G(p)$.

In addition, if $\Gamma(p)$ is the discrete co-compact subgroup of $G(p)$ given by (\ref{Gamma-p})
then, using (\ref{h-0-h-1}), (\ref{h-2}), (\ref{Gamma-N}) and (\ref{Gamma-p}), we obtain that
\begin{equation}\label{projec-discrete}
\pi(\Gamma(p)) = (\Gamma_{\R^2}\rtimes_{\zeta} \Z) \times z_2 \Z.
\end{equation}
Thus, from (\ref{Gamma-K}), (\ref{discrete-base}),  (\ref{description-base-space}) and (\ref{projec-discrete}), we have that the Lie group epimorphism $\pi: G(p) \to K \times \R$ induces a principal $S^1$-bundle
\[
\hat{\pi}: \widehat{G(p)} \to \hat{K} \times S^1
\]
where the action of $S^1$ on $\widehat{G(p)}$ is just the flow of the Reeb vector field $\hat{\xi}$ and, therefore, the base space is just the orbit space of $\hat{\xi}$.

On the other hand, using (\ref{left-inv-1-forms-G(p)}), (\ref{structure-equations}), (\ref{left-inv-1-forms-K}), (\ref{Omega}) and (\ref{pi}), it follows that $\pi^*(\Omega) = d\eta$ and, consequently,
\[
\hat{\pi}^*(\hat{\Omega}) = d\hat{\eta}.
\]
This implies that our contact structure $\hat{\eta}$ on $\widehat{G(p)}$ is regular.
\end{proof}
\begin{remark}
The left-invariant $2$-form $\Omega$ on $K\times\R$ was used, in \cite{FeGr}, in order to obtain examples of Lefschetz symplectic structures on compact solvmanifolds of the form $(K\times\R) / D$, with $D$ a co-compact discrete subgroup. We remark that these manifolds are formal (see \cite{FeGr} for more details).
\end{remark}

\section{The $K$-contact structure on $\widehat{G(p)}$ and its properties}
In this section, we will introduce a $K$-contact structure on $\widehat{G(p)}$ and we will discuss its properties.

Denote by $\hat{A}$, $\hat{B}$, $\hat{U}$, $\hat{R}$ and $\hat{\xi}$ the vector fields on $\widehat{G(p)}$ which are induced by the left-invariant vector fields $A$, $B$, $U$, $R$ and $\xi$, respectively. Then, it is clear that $\{ \hat{A}, \hat{B}, \hat{U}, \hat{R}, \hat{\xi} \}$ is a global basis of vector fields on $\widehat{G(p)}$. Moreover, the dual basis of $1$-forms is just
$\{ \hat{\alpha}, \hat{\beta}, \hat{\theta}, \hat{\gamma}, \hat{\eta}\}$.

In the following theorem we will use the following notions.
Given a commutative differential graded algebra (CDGA) $(A,d)$
we define an \emph{elementary extension} of $(A,d)$ with respect to the closed
element $b\in A_{2}$, to be  $A\otimes \Lambda \left\langle y
\right\rangle$, with $\deg(y)=1$ and
$d y = b$.
We will say that a
contact manifold $(M,\eta)$ is of \emph{Tievsky type}, if the de Rham cohomology
algebra $\Omega^*(M)$ is quasi-isomorphic (as CDGA) to the elementary extension of
$H_B(M)$ with respect to $[d \eta]_B$.
Note, that according to~\cite{Ti}, every Sasakian manifold is of Tievsky
type.
\begin{theorem}
Let $\hat{\phi}$ be the $(1, 1)$-tensor field on $\widehat{G(p)}$ which is characterized by the following conditions:
\begin{equation}\label{phi-hat}
\hat{\phi}(\hat{A}) = -\hat{B}, \; \;  \hat{\phi}(\hat{B}) = \hat{A}, \; \;   \hat{\phi}(\hat{U}) = \hat{R},  \; \; \hat{\phi}(\hat{R}) = -\hat{U}, \;\; \hat{\phi}({\hat{\xi}}) = 0,
\end{equation}
and $\hat{g}$ the Riemannian metric on $\widehat{G(p)}$ whose matrix associated with respect to the basis $\{\hat{A}, \hat{B}, \hat{U}, \hat{R}, \hat{\xi} \}$ is the identity, that is, the basis $\{\hat{A}, \hat{B}, \hat{U}, \hat{R}, \hat{\xi} \}$ is orthonormal. Then:
\begin{enumerate}[(1)]
\item
The almost contact metric structure $(\hat{\phi}, \hat{\xi}, \hat{\eta}, \hat{g})$ on $\widehat{G(p)}$ is $K$-contact.

\item
The contact structure $\hat{\eta}$ on $\widehat{G(p)}$ has the Hard Lefschetz property.

\item The minimal model (over $\R$) of $\widehat{G(p)}$ is isomorphic to the
exterior algebra $\Lambda\langle a_1,a_2,b\rangle$ with zero differential, where the  grading
is determined by the requirement that $a_1$, $a_2$ have degree $1$, and
$b$ has degree $3$. In particular, $\widehat{G(p)}$ is formal.
\item
	The contact manifold $(\widehat{G(p)},\eta)$ is of Tievsky type.

\item
The manifold $\widehat{G(p)}$ does not admit Sasakian structures.
\end{enumerate}
\end{theorem}
\begin{proof}
{\em (1)} From (\ref{structure-equations}) and(\ref{phi-hat}), we deduce that $(\hat{\phi}, \hat{\xi}, \hat{\eta}, \hat{g})$ is a contact metric structure on $\widehat{G(p)}$. In addition, using (\ref{structure-eqs-G(p)}), we have that the element in ${\mathfrak g}$, which induces the left-invariant vector field $\xi$ on $G(p)$, belongs to the center of ${\mathfrak g}$. So, the Reeb vector field $\hat{\xi}$ is Killing with respect to the Riemannian metric $\hat{g}$.

{\em (2)} A direct computation, using (\ref{deRham-compact-solv}) and the fact that $d\hat{\eta} = -\hat{\alpha}\wedge \hat{\beta} - \hat{\gamma} \wedge \hat{\theta}$, proves that the Lefschetz relation in degree $1$ (resp. degree $2$) for the compact contact manifold $(\widehat{G(p)}, \hat{\eta})$ is the graph of an isomorphism between $H^1(\widehat{G(p)})$ and $H^4(\widehat{G(p)})$ (resp. $H^2(\widehat{G(p)})$ and $H^3(\widehat{G(p)})$). This proves {\em (2)}.

{\em (3)} Define the map $\tau \colon \Lambda\langle a_1,a_2,b\rangle \to
\Omega^*(M)$ by
\begin{equation*}
	\tau(a_1) = \hat{\gamma},\qquad  \tau(a_2) = \hat{\theta} ,\qquad  \tau(b)=
	\hat{\eta} \wedge (\hat{\theta}\wedge \hat{\gamma} - \hat{\alpha} \wedge
	\hat{\beta}),\qquad
\end{equation*}
where the $1$-forms $\hat{\alpha}$, $\hat{\beta}$, $\hat{\gamma}$,
$\hat{\theta}$ are defined in Subsection~\ref{derham}.
Since, the forms $\hat{\gamma}$, $\hat{\theta}$, and
$\hat{\eta} \wedge (\hat{\theta}\wedge \hat{\gamma} - \hat{\alpha} \wedge
\hat{\beta})$ are closed, the map $\tau$ is a well-defined homomorphism
of CDGAs. It follows from Corollary~\ref{derham-cor} that $\tau$ is a
quasi-isomorphism.
Since $\Lambda\left\langle a_1,a_2,b \right\rangle$ is a free CDGA with zero
differential, it is a minimal Sullivan algebra. Moreover, it is quasi-isomorphic
to its cohomology, and thus is formal.

{\em (4)}
We will show that there is a quasi-isomorphism from $\Lambda\left\langle
a_1,a_2,b
\right\rangle$ into $H_B(M)\otimes \Lambda\left\langle y \right\rangle$
with $dy := [d\eta]_B$.
Since $\xi$
is regular, we get that $H_B(\widehat{G(p)})$ is isomorphic as graded algebra to
$H((K/\Gamma_K)\times S^1)$.
The cohomology ring of $H((K/\Gamma_K)\times S^1)$ was computed
in~\cite{FeGr} and it is isomorphic to
\begin{equation*}
	\Lambda\left\langle u,v \right\rangle \otimes \R[w]/ (w^2)
\end{equation*}
with elements $u$, $v$ of degree $1$ and $w$ of degree $2$.
The element $[d\eta]_B = [\hat{\Omega}]$ is given by $-uv - w$  in this description.
We define the map
$\rho$ from $\Lambda\left\langle a_1,a_2,b \right\rangle$ to
\[
A:= H( (K/\Gamma_K)\times S^1 ) \otimes \Lambda\left\langle y \right\rangle \quad \mbox{(with $dy= -uv-w$),}
\]
by
\begin{equation*}
	\rho(a_1) = u, \qquad \rho(a_2) = v, \qquad \rho(b) = (uv - w) y.
\end{equation*}
The elements $u$ and $v$ are closed in $A$. Moreover,
\begin{equation*}
	d( (uv- w) y) = - (uv-w) (uv + w) = - uv w + w uv = 0.
\end{equation*}
Thus the map $\rho$ is well-defined. It is easy to check that $\rho$ induces an
isomorphism in the cohomology. Thus $A$ is quasi-isomorphic to the de Rham
algebra $\Omega^*(\widehat{G(p)})$, which shows that $\widehat{G(p)}$ is a
manifold of Tievsky type.

{\em (5)} Recall that $\widehat{G(p)}=G(p)/\Gamma(p)$ is a compact solvmanifold with fundamental group $\Gamma(p)$. Now, suppose that $\widehat{G(p)}$ admits a Sasakian structure.
It is well known that the fundamental group of a compact solvmanifold is polycyclic, thus $\Gamma(p)$ is polycyclic. In \cite[Corollary 1.3]{Ka}, Kasuya shows that if the fundamental group of a compact Sasakian manifold is polycyclic, then it is in fact virtually nilpotent, i.e. it contains a nilpotent subgroup $\Gamma'$ of finite index. As $G(p)$ is completely solvable and $\Gamma'$ is a co-compact nilpotent subgroup of $G(p)$, we get by Sait{\^o}'s rigidity theorem \cite{Sa} that $G(p)$ must be nilpotent. But this contradicts Remark~\ref{rmk:non-nilpotent}.
\end{proof}

\end{document}